\documentclass[twoside,reqno,12pt]{amsart}
\usepackage[left=1in,right=1in,top=1in,bottom=1in]{geometry}

\usepackage[colorlinks=true]{hyperref}

\usepackage{amsfonts}
\usepackage{amsmath, amssymb}
\usepackage{color}
\usepackage{xcolor}
\usepackage{comment}
\usepackage{enumerate}
\newcommand{\HOX}[1]{\marginpar{\footnotesize sharp1}}

\newtheorem{thm}{Theorem}[section]
\newtheorem{cor}[thm]{Corollary}
\newtheorem{lem}[thm]{Lemma}
\newtheorem{prop}[thm]{Proposition}

\newtheorem{defn}[thm]{Definition}
\newtheorem{rem}[thm]{Remark}

\theoremstyle{definition}

\numberwithin{equation}{section}

\newcommand{\thmref}[1]{Theorem~\ref{sharp1}}
\newcommand{\secref}[1]{\S\ref{sharp1}}
\newcommand{\lemref}[1]{Lemma~\ref{sharp1}}

\renewcommand{\div}{\operatorname{div}}

\renewcommand{\t}[1]{\tilde{sharp1}}

\newcommand{\tr}{\hbox{tr}\,}

\newcommand{\dbtilde}[1]{\tilde{\tilde{sharp1}}}

\def\tilde{\widetilde}
\def \bfo {\begin {eqnarray*} }
\def \efo {\end {eqnarray*} }
\def \ba {\begin {eqnarray*} }
\def \ea {\end {eqnarray*} }
\def \beq {\begin {eqnarray}}
\def \eeq {\end {eqnarray}}
\def \supp {\hbox{supp }}

\def \det {\hbox{det}}

\def\tilde{\widetilde}
\def \bfo {\begin {eqnarray*} }
\def \efo {\end {eqnarray*} }
\def \ba {\begin {eqnarray*} }
\def \ea {\end {eqnarray*} }
\def \beq {\begin {eqnarray}}
\def \eeq {\end {eqnarray}}
\def \supp {\hbox{supp }}

\def \det {\hbox{det}}

\usepackage{mathtools}

\usepackage{color,comment}

\begin{document}
\title[Calder\'on problem with Magnetic Perturbation]{Riemannian and Lorentzian Calder\'on Problem under Magnetic Perturbation}

\author[]{Yuchao Yi}
\address{}
\email{}

\begin{abstract}
We study both the Riemannian and Lorentzian Calder\'on problem when a family of Dirichlet-to-Neumann maps are given for an open set of magnetic/electromagnetic potentials. For the Riemannian version, by allowing small perturbations of the magnetic potential, we use the Runge Approximation Theorem to show that the metric can be uniquely determined. There is no gauge equivalence in this case. For the Lorentzian version, we use microlocal analysis to construct the trajectory of null-geodesics via generic perturbations of the electromagnetic potential, hence the conformal class of the metric can be constructed. Moreover, we also show, in the Lorentzian case, the same result can be obtained using generic perturbations of the metric itself.
\end{abstract}

\maketitle
\footnotetext{Department of Mathematics, University of California San Diego. \texttt{Email: yuyi@ucsd.edu}}

\section{Introduction}

Suppose $(M, g)$ is a smooth, connected, compact Riemannian manifold with boundary of dimension $n \geq 3$. Let $\Delta_g$ be the Laplace-Beltrami operator, then the Dirichlet-to-Neumann map (DN map) is defined as
\begin{equation}
    \Lambda_g: C^{\infty}(\partial M)\ni f \mapsto \partial_{\nu}u|_{\partial M} \in C^{\infty}(\partial M)
\end{equation}
where $u$ is the unique solution of
\begin{equation}\label{eq: laplace}
    \Delta_g u = 0 \quad \text{in } M^{\circ}, \quad u|_{\partial M} = f.
\end{equation}
Here $M^\circ$ is the interior of $M$. There exists a natural gauge equivalence for the DN map when $n \geq 3$. Namely, if $\Psi: M \to M$ is a diffeomorphism such that $\Psi|_{\partial M}$ is the identity map, then
\begin{equation}
    \Lambda_g = \Lambda_{\Psi^*g}.
\end{equation}
The well-known anisotropic Calder\'on problem asks if the metric $g$ is determined by the DN map $\Lambda_g$ up to boundary fixing isometry (see \cite[Conjecture A]{LU89}), we also refer to it as the Riemannian Calder\'on problem in this paper. The problem is widely open in dimension $n \geq 3$, see Section \ref{subsec: history} for some history of the problem.

In this work, we consider the magnetic Laplace-Beltrami operator $\Delta_{g,A}$ where $A$ is a smooth 1-form. In local coordinates, it has the form
\[
\Delta_{g,A} = |g|^{-\frac{1}{2}}(\partial_j-iA_j)|g|^{\frac{1}{2}}g^{jk}(\partial_k - iA_k).
\]
The 1-form $A$ is usually referred to as the magnetic potential. It appears frequently in physics, for example when a charged quantum particle moves in a magnetic field. Denote the DN map
\[
\Lambda_{g, A}: f \mapsto (\partial_{\nu}-i\langle A, \nu \rangle)u|_{\partial M}
\]
where $u$ is the unique solution of
\begin{equation}
    \Delta_{g,A}u = 0 \quad \text{in } M^{\circ}, \quad u|_{\partial M} = f.
\end{equation}
In this case, if $\Psi: M \to M$ is a boundary fixing diffeomorphism, and $d\omega$ is a smooth exact 1-form such that $\omega|_{\partial M} = 0$, then
\begin{equation}
    \Lambda_{g,A} = \Lambda_{\Psi^*g, \Psi^*A+d\omega}.
\end{equation}
This is referred to as the natural gauge equivalence. Specifically, if $u$ is the solution with respect to $g$ and $A$, then $e^{i\omega}\Psi^*u$ is the solution with the same Dirichlet boundary data respect to $\Psi^*g$ and $\Psi^*A + d\omega$.

In this work, we show that $g$ can be fully determined without any gauge equivalence if the DN map with magnetic potential is given for all sufficiently small magnetic potentials.

\begin{thm}\label{thm: riemannian perturb A}
   Let $M$ be a smooth connected compact manifold with boundary of dimension $n \geq 3$, and $g_1, g_2$ are two Riemannian metrics on $M$. Suppose $\Lambda_{g_1,A} = \Lambda_{g_2, A}$ for all $A$ in a neighborhood of 0 with respect to $C^k$ topology for some $k \geq 0$, then $g_1 = g_2$.
\end{thm}

The following result is a direct consequence of Theorem \ref{thm: riemannian perturb A}.
\begin{cor}
    For $j = 1, 2$, let $(M_j, g_j)$ be a smooth connected compact Riemannian manifold with boundary of dimension $n \geq 3$. Suppose there exists a boundary fixing diffeomorphism $\Psi: M_1 \to M_2$ such that $\Lambda_{g_1,A} = \Lambda_{g_2, (\Psi^{-1})^*A}$ for all smooth 1-form $A$ on $M_1$ in a neighborhood of 0 with respect to $C^k$ topology for some $k \geq 0$, then $g_1 = \Psi^*g_2$.
\end{cor}

The same problem can be considered when the metric is Lorentzian. Consider a Lorentzian manifold $(M, g)$ of dimension $n \geq 3$ with a timelike boundary $\partial M$. We say the manifold is (lightlike) non-trapping if any inextendible lightlike geodesic $\gamma: [a, b] \to M$ eventually leaves $M$ in finite time, that is, $-\infty < a \leq b < \infty$ and $\gamma(a), \gamma(b) \in \partial M$. For simplicity, we assume the manifold is non-trapping and admissible (see Definition \ref{def: admissible}), in particular, we consider the following electromagnetic wave equation
\begin{equation}\label{eq: general wave equation}
    \begin{cases}
        \Box_{g,A,q} u = 0 \quad \text{in } M^{\circ},\\
        u|_{\partial M} = f,\\
        \supp(u) \subset J^+ (\supp(f)),
    \end{cases}
\end{equation}
where $J^+$ denotes the causal future, and $\Box_{g,A,q}$ is the wave operator with first order perturbation
\[
\Box_{g,A,q} = |g|^{-\frac{1}{2}}(\partial_j - iA_j)|g|^{\frac{1}{2}}g^{jk}(\partial_k-iA_k) + q.
\]
Here $A$ is a smooth 1-form and $q$ is a smooth function. We emphasize that in the Lorentzian case, for example in Lorentzian formulation of electrodynamics, $A$ is almost always the electromagnetic potential instead of the magnetic potential. It satisfies $dA = F$ where $F$ is the electromagnetic field tensor. In this work, in the Lorentzian case we shall refer to $A$ as the electromagnetic potential.

The wave equation \eqref{eq: general wave equation} is well-posed for $f \in H^1_{loc}(\partial M)$ by \cite[Theorem 24.1.1]{Hor85III}. We define the DN map as
\begin{equation}
    \Lambda_{g,A,q}:  H^1_{loc}(\partial M) \ni f \mapsto (\partial_{\nu}-i\langle A, \nu\rangle) u|_{\partial M} \in  \bar{L}^2_{loc}(M^{\circ})
\end{equation}
where $\nu$ is the unit outward normal vector on the boundary, and $\bar{L}^2_{loc}(M^{\circ})$ refers to the space of restrictions to $M^{\circ}$ of locally $L^2$ functions on $M$ (see \cite[Theorem 24.1.1]{Hor85III} and \cite[Page 113]{Hor85III}). Again if $\Psi: M \to M$ is a boundary fixing diffeomorphism and $d\omega$ is a smooth exact 1-form such that $\omega|_{\partial M} = 0$, then
\begin{equation}
    \Lambda_{g,A,q} = \Lambda_{\Psi^*g, \Psi^*A+d\omega, \Psi^*q}.
\end{equation}
The existence of the potential term allows for another gauge equivalence. Namely, let $\phi$ be a smooth function on $M$ such that $\phi|_{\partial M} = \partial_{\nu}\phi|_{\partial M} = 0$, then
\begin{equation}
    \Lambda_{g,A,q} = \Lambda_{e^{-2\phi}g, A+d\omega, e^{2\phi}(q-q_{\phi})}, \quad q_{\phi} = e^{\frac{n-2}{2}\phi}\Box_ge^{\frac{2-n}{2}\phi}.
\end{equation}
We refer to \cite{SY18} for details of the gauge equivalences.

The Lorentzian Calder\'on problem asks if the metric $g$ is uniquely determined by the DN map $\Lambda_{g, A, q}$ up to these natural gauges. Suppose the electromagnetic field is allowed to perturb, namely we have the information of $\Lambda_{g,A,q}$ for a family of different $A$. In this case, one can use this perturbation to reconstruct the conformal class of the metric $g$. 

\begin{thm}\label{thm: lorentzian perturb A}
    Suppose $(M, g)$ is non-trapping and admissible. Let $\mathcal{B}_{\delta}(\mathcal{A})$ be the set of smooth 1-forms that are $\delta$ close to some fixed smooth 1-form $\mathcal{A}$ in the $C^k$ topology for some $k$. For any $U \subset M^{\circ}$ open, denote the set of perturbations of $\mathcal{A}$ in $U$ as
    \[
    \mathcal{B}_{\delta}(\mathcal{A}; U) = \{A \in \mathcal{B}_{\delta}(\mathcal{A}): \supp(A - \mathcal{A}) \subset U\}.
    \]
    If $L$ is the lens relation (see Definition \ref{def: lens relation}) and $L(x', \xi') = (y', \eta')$, then the unique null-geodesic corresponding to $(y', \eta'; x', \xi')$ is given by
    \begin{equation}\label{eq: determine gamma from A}
        \begin{split}
            \gamma = \{x \in M^{\circ}: \forall &U \ni x \text{ open, } \exists A \in \mathcal{B}_{\delta}(\mathcal{A}; U) \text{ s.t. }\\
            &\sigma(\Lambda_{g,A,q})|_{(y', \eta'; x', \xi')} \neq \sigma(\Lambda_{g,\mathcal{A},q})|_{(y', \eta'; x', \xi')}\}.
        \end{split}
    \end{equation}
    Moreover, the set of such $A$ is generic in $\mathcal{B}_{\delta}(\mathcal{A}; U)$. In particular, the conformal class of $g$ can be constructed from $\{\Lambda_{g,A,q}: A \in \mathcal{B}_{\delta}(\mathcal{A})\}$.
\end{thm}

\begin{rem}
    For the definition of principal symbol $\sigma(\Lambda_{g,A,q})$, see Section \ref{sec: lorentzian preliminary}. A Lorentzian manifold is called non-trapping if any inextendible null-geodesic has two end points in the boundary. When the manifold is non-trapping and admissible (see Definition \ref{def: admissible}), all null-geodesics in the interior would eventually reach the boundary transversally, which makes the computation easier. We emphasize that the idea of the proof works for more general Lorentzian manifolds:
    \begin{enumerate}
        \item For example, the strictly convexity can be relaxed, then besides Proposition \ref{prop: principal symbol formula}, one also needs to analyze the propagation of singularity along gliding rays or diffractive rays. 
        \item The non-trapping condition can also be removed, and one can only construct the null-geodesics that are non-trapping.
        \item If the DN map is the partial data version as in \cite{YZ25}, then one can construct all the broken null-geodesics that travels from source region to observation region.
    \end{enumerate}
    Finally, just like Theorem \ref{thm: riemannian perturb A}, since we work on a fixed base manifold, the perturbation of $A$ eliminates the possibility of having boundary fixing isometry as a gauge equivalence.
\end{rem}

Intuitively, the theorem states that perturbations of the electromagnetic potential in a test region can help determine whether the null-geodesic passes through it or not. We emphasize that the theorem is robust in the sense that if $x \in \gamma$, then a generic local perturbation of the electromagnetic potential in a neighborhood of $x$ would reveal this fact. Moreover, one does not need to have any a priori information about $\mathcal{A}$, the only thing required is the ability to add perturbation at various locations. It is worth mentioning that if $A - \mathcal{A} = d\omega$ and $\omega|_{\partial M} \equiv 0$, then by the gauge equivalence the DN map does not change. In fact the solution simply changes from $u$ to $e^{i\omega}u$. In physics, this corresponds to $\mathcal{A} + d\omega$ and $\mathcal{A}$ represent the same electromagnetic field tensor $F$. On the other hand, they still represent the same electromagnetic field $F$ if $A - \mathcal{A}$ is a closed 1-form, but we no longer have the simple change in solutions as in the exact case. As a result, a change in closed 1-form may be used to detect the trajectory of null-geodesics. Specifically, in the proof we will show that the criterion is given by the shift in phase, in physics this is closely related to the well-known phenomenon called the Aharonov–Bohm effect, we refer to Remark \ref{rem: ab effect} for more details. See also \cite{AB59, AB61}.

Besides perturbing the electromagnetic potential, a similar problem can be considered when the perturbation is with respect to other parameters in the equation. In particular, we study the case where the metric itself is allowed to perturb. Consider the wave equation
\begin{equation}\label{eq: wave}
    \begin{cases}
        \left(\Box_g+q_g\right)u = 0\\
        u|_{\partial M} = f\\
        \supp(u) \subset J^+ (\supp(f))
    \end{cases}
\end{equation}
where the potential term $q_g$ may depend on $g$ locally. That is, $q_g(x)$ depends on $g$ only on an arbitrarily small neighborhood of $x$. For instance, the potential term in the Yamabe equation
\begin{equation}\label{eq: Yamabe}
    \left(\Box_g-\frac{n-2}{4(n-1)}R_g\right)u = 0
\end{equation}
is related to scalar curvature $R_g$, which depends on $g$. Simple computation shows that a conformal change of $g$ which is the identity map on the boundary does not change the DN map. On the other hand, the Klein-Gordon equation
\begin{equation}\label{eq: Klein-Gordon}
    \left(\Box_g+m^2\right)u = 0
\end{equation}
provides an example where the potential term is independent of $g$. We show that the conformal class of the metric can also be recovered using perturbation of the metric.

\begin{thm}\label{thm: lorentzian perturb g}
    Suppose $(M, g)$ is non-trapping and admissible, denote $\Lambda_{g}$ the DN map of \eqref{eq: wave} with $q_g$ depending on $g$ locally. If $L$ is the lens relation (see Definition \ref{def: lens relation}) and $L(x', \xi') = (y', \eta')$, then the unique null-geodesic associated with $(y', \eta'; x', \xi')$ can be constructed by the following equivalence conditions.
    \begin{enumerate}
        \item $\gamma$ passes through $x \in M^{\circ}$.
        \item For any $U \ni x$ open neighborhood and conormal distribution $f$ singular at $(x', \xi')$, there exists smooth symmetric 2-tensor $h$ compactly supported in $U$ such that $(y', \eta')$ is in the wavefront set of $\partial_{\epsilon}|_0\Lambda_{g-\epsilon h}(f)$.
    \end{enumerate}
    Moreover, the set of such $h$ is generic in the space of smooth symmetric 2-tensors compactly supported in $U$ with respect to $C^k$ topology for any $k$. In particular, let $\mathcal{B}_{\delta}(g)$ be the set of smooth Lorentzian metrics that are $\delta$ close to $g$ in $C^k$ topology for some $k$, then the conformal class of $g$ can be constructed from $\{\Lambda_{g'}: g' \in \mathcal{B}_{\delta}(g)\}$.
\end{thm}

Clearly, one actually recovers the conformal class of all the metrics in $\mathcal{B}_{\delta}(g)$.

\subsection{History}\label{subsec: history}
Both the Riemannian Calder\'on problem and the Lorentzian Calder\'on problem are widely open for dimension $n \geq 3$.

The Calder\'on problem originated from the Electrical Impedance Tomography (EIT) problem considered by Calder\'on in \cite{Cal80}. EIT asks the following: given a bounded domain $\Omega\subset\mathbb{R}^n$ with an unknown (possibly anisotropic) conductivity tensor $\gamma(x)$, can one recover $\gamma$ from the Dirichlet-to-Neumann map which sends imposed voltages on $\partial\Omega$ to the resulting boundary currents on $\partial \Omega$?  Equivalently, the DN map encodes exactly the voltage-to-current measurements at the boundary. In the anisotropic setting and for $n\geq3$, this inverse boundary value problem is precisely Calder\'on’s original formulation of recovering the conductivity tensor from boundary data (up to the natural obstruction of boundary fixing diffeomorphisms). For the 2-dimensional Calder\'on problem, besides the boundary fixing isometry, a boundary fixing conformal change is also a gauge equivalence. The dimension 2 case has been fully solved in \cite{LU01}, but the higher dimension Calder\'on problem remains one of the central open questions in the theory of inverse problems. For more conductivity results, we refer to \cite{SU87, KV84, KV85, HT13, AP06}; and we refer to \cite{Sal13, Uhl14} for more details on the history.


When $n \geq 3$ and the metric is real analytic, the problem is solved in \cite{LTU03, LU01, LU89}. In \cite{LLS20}, Lassas, Liimatainen and Salo proved the same result using a different approach. They used Poisson embedding to identify interior points with distributions on the boundary. When the metric is only $C^{\infty}$, in \cite{DKSU09} and \cite{DKLS16}, it was shown that a metric in a fixed conformal class is uniquely determined by boundary measurements if the metric is conformally transversally anisotropic.

The Lorentzian Calder\'on problem is relatively new comparing to the Riemannian version, but there has been some important progress in the past decade. One of the most important results is the recovery of the potential term in \cite{AFO22} and \cite{AFO24}. In \cite{AFO22}, Alexakis, Feizmohammadi and Oksanen proved a type of unique continuation principle when the metric is given under some geometric assumptions, and used it to determine the potential term. When $A = 0$ and the conformal class of $g$ is given, determining $g$ reduces to determining the potential term $q$. As a result, they recovered the conformal factor under some a priori conditions on the metric. In \cite{AFO24}, they proved the same result for metrics sufficiently close to the Minkowski metric. For other results related to the recovery of potential term, see \cite{Ste89, FIKO21, Isa91, RS91}. Note that both of Theorem \ref{thm: lorentzian perturb A} and Theorem \ref{thm: lorentzian perturb g} recover the conformal class of the metric, one can then combine with results from \cite{AFO22} and \cite{AFO24} to recover the conformal factor.

In \cite{Esk07, Esk17}, Eskin proved the Lorentzian Calder\'on problem when the metric and potential are both real analytic in time. The main tool used there is Boundary Control method, which was introduced by Belishev in \cite{Bel87} and is closely related to the Unique Continuation Principle, see \cite{BK92, Bel07, Tat99, RZ98} for details.

As for the recovery of conformal class of the metric, one important result comes from when the equation has non-linearity. Specifically, in \cite{KLU18}, Kurylev, Lassas and Uhlmann proved that the conformal class of the metric can be determined when the wave equation is non-linear and of the form
\[
\Box_gu + a(x)u^2 = 0.
\]
They pioneered the use of higher order linearization and microlocal analysis to generate earliest light observation sets, which recovers the conformal class of the metric. The earliest light observation set is essentially the shape of the light cone intersected with the observation set. The method highly relies on the non-linearity, as it allows the waves to interact with each other and create a collision point in the interior. The same idea also helps recover the non-linearity itself, and the study of non-linear hyperbolic inverse problems has been an active research area ever since, for example see \cite{UZ23, HUZ22, CLOP22, CLOP21, OSSU24, LUW18, Yi24, SUW22, UZ22} and the surveys \cite{Las25,UZ21}. This approach, however, does not work for linear equations, since two waves do not interact with each other.

Finally, there has been some boundary determination results. In \cite{SY18}, Stefanov and Yang proved that the local DN map is a pseudodifferential operator, which stably recovers the jet bundle of $g, A$ and $q$ on the boundary. They also proved that away from the diagonal, the DN map is a Fourier Integral Operator with canonical relation being the lens relation (see Section \ref{sec: lorentzian preliminary}). Given the metric, they proved stable recovery of lens relation and light ray transform of $A$ and $q$ on the boundary. In \cite{YZ25}, the author and Zhang proved boundary determination under partial data assumption using the Fourier Integral Operator part of the DN map. Specifically, the Dirichlet initial data is supported on a fixed region $U$, while the observation is made on a fixed subset of the boundary $V$, and $U$ and $V$ are disjoint from each other. In both \cite{SY18} and \cite{YZ25}, there is no interior recovery of the metric.

To the author's knowledge, metric recovery from a family of DN maps indexed by the perturbation of parameters appears to be new in both Riemannian and Lorentzian settings. Classical Calder\'on results typically assume a single DN map and focus on boundary determination or single-map interior uniqueness. Almost all of the existing results are obtained under additional structural assumptions of the metric (for example, analyticity \cite{LU01}, CTA structures \cite{DKSU09, DKLS16}, or curvature bounds \cite{FIKO21}), and many of them require the metrics to be in the same conformal class. In contrast, we add no such assumptions on the metric structure, but instead work with an active, coefficient-parametric data model: an open family of magnetic or electromagnetic perturbations and the associated family of DN maps.

\subsection{Idea of the proofs}
We now briefly discuss the ideas behind the proofs and make comparison between the Riemannian and Lorentzian versions. As mentioned before, one of the key points in recovering the interior metric is the ability to focus on an interior point. Non-linearity and microlocal analysis in the Lorentzian setting provide this ability via wave interactions and propagation of singularities. However, in a linear equation the solutions do not interact, and in the Riemannian case there is no propagation.

To solve the first problem, note that even though the operator is linear, its dependence on the parameters may be non-linear. Specifically, in our case, both $\Delta_{g,A}$ and $\Box_{g,A,q}$ depend non-linearly on $A$ and on $g$ itself, which provides the possibility of obtaining focused information around interior points via local perturbation. For the Lorentzian case, we again use microlocal analysis, which allows us to focus on a single null-geodesic and analyze the propagation of singularities. If the null-geodesic $\gamma$ does not travel through some interior point $x$, then a smooth perturbation in a sufficiently small neighborhood of $x$ should not affect the propagation of the singularity along $\gamma$. The main work lies in showing that if $\gamma$ does pass through $x$, then some particular perturbation around $x$ would affect the propagation of the singularity, and that this change can be picked up from the boundary observation. We actually show that such perturbations are generic. Together, they provide a criterion for determining whether $\gamma$ passes through a point, thereby recovering the trajectory of the null-geodesic.

In Theorem \ref{thm: lorentzian perturb A}, we perturb the electromagnetic potential to probe the propagation of particles along null-geodesics. We show in the proof that this only affects the spin of the particles, and hence the trajectory can be recovered by detecting changes in particle spin. In Theorem \ref{thm: lorentzian perturb g}, we use the perturbation of the metric itself. The trajectory of the chosen null-geodesic may change in this case. One needs to show that this indeed happens for some perturbation, and this change can be detected from boundary measurements.

The same method obviously does not work for the Riemannian case, since propagation is a feature specific to hyperbolic equations. In order to solve this problem for the Riemannian setting, we turn to the Runge Approximation Theorem, which is only enjoyed by elliptic and parabolic equations. Again, even though there is no propagation, we can use the non-linear dependence on $A$ and local perturbation to focus around an interior point. We show in Theorem \ref{thm: riemannian perturb A} that this focusing gives us information of the gradient of the solution, and turns the problem into determining whether the following transformation gives a gauge equivalence.

Suppose $B$ is an isomorphism on $T^*M$, and up to a diffeomorphism we may assume $B: T^*_xM \to T^*_xM$ for all $x \in M$. Suppose $B$ is the identity map on the boundary, and the boundary is connected. Consider some Riemannian metric $g$, and suppose
\begin{equation}
    \Delta_g v = 0 \implies B(dv) \text{ is exact.}
\end{equation}
Consider the new metric given in local coordinates by
\begin{equation}
    \tilde{g}^{ij} = \det(B)^{\frac{1}{n-2}}(B^{-1})^j_kg^{ik}.
\end{equation}
For every $v$ such that $\Delta_g v = 0$, $B(dv) = du$ for some $u$, and $dv = du$ on boundary. By adding a constant, we may assume $u|_{\partial M} = v|_{\partial M}$. Since by construction $g$ and $\tilde{g}$ also agree on the boundary, they have the same unit outward normal vector $\nu$, so $\partial_{\nu}u|_{\partial M} = \partial_{\nu}v|_{\partial M}$. Finally, by construction,
\begin{equation}
    \Delta_{\tilde{g}}u = |\tilde{g}|^{-\frac{1}{2}}\partial_j(|\tilde{g}|^{\frac{1}{2}}\tilde{g}^{jk}\partial_ku) = |\tilde{g}|^{-\frac{1}{2}}\partial_j(|\tilde{g}|^{\frac{1}{2}}\tilde{g}^{jk}B_k^i\partial_iv) = \frac{|g|^{\frac{1}{2}}}{|\tilde{g}|^{\frac{1}{2}}}\Delta_gv = 0.
\end{equation}
Hence $B$ gives a gauge equivalence $\Lambda_g = \Lambda_{\tilde{g}}$. We use a Runge Approximation type of result to show that the only such transformation in the Riemannian setting is the identity map, which proves Theorem \ref{thm: riemannian perturb A}.

\subsection{Outline of the paper}
The structure of the paper is as follows.

The Riemannian version is studied in Section \ref{sec: Riemannian}. In Section \ref{sec: Runge Approximation Theorem}, we study the aforementioned transformation; in Section \ref{sec: proof of riemannian thm}, we prove Theorem \ref{thm: riemannian perturb A}.

The Lorentzian version is studied in Section \ref{sec: Lorentzian}. In Section \ref{sec: lorentzian preliminary}, we include the necessary tools and results; in Section \ref{sec: proof of theorem 1}, we prove Theorem \ref{thm: lorentzian perturb A}; in Section \ref{sec: proof of theorem 2}, we prove Theorem \ref{thm: lorentzian perturb g}.

\section*{Acknowledgment}
The author would like to thank Gunther Uhlmann for reading the draft version of this paper and for giving helpful comments.

\section{Riemannian Calder\'on problem}\label{sec: Riemannian}

\subsection{Preliminaries}\label{sec: Runge Approximation Theorem}

We first include a Runge Approximation type of result which will be used later, for a proof see for example the Appendix of \cite{LLS20}.
\begin{prop}[{\cite[Proposition A.5]{LLS20}}]\label{prop: runge}
    Let $\Gamma$ be a non-empty open subset of $\partial M$. Let $p \in M^{\circ} \cup \Gamma$, let $a_0 \in \mathbb{R}$, let $\xi_0 \in T_p^*M$, and let $H_0$ be a symmetric 2-tensor at $p$ satisfying $\tr_g(H_0) = 0$. There exists $f \in C_c^{\infty}(\Gamma)$ such that the solution of
    \begin{equation}
        -\Delta_gu = 0 \quad \text{in } M, \quad u|_{\partial M} = f
    \end{equation}
    satisfies $u(p) = a_0$, $du|_p = \xi_0$ and $\text{Hess}_g(u)|_p = H_0$.
\end{prop}

It is easy to show that if an isomorphism $B$ on $T^*M$ maps all exact 1-forms to closed 1-forms, then up to a diffeomorphism on $M$, $B$ has to be a locally constant scaling. By composing with a diffeomorphism we may assume $B$ maps each $T_x^*M$ to $T_x^*M$. If $u$ is a smooth function on $M$, then both $du$ and $udu = \frac{1}{2}d(u^2)$ are exact 1-forms. By the assumption, we have
\begin{equation}
    0 = d(B(udu)) = d(uB(du)) = du \wedge B(du).
\end{equation}
Since this holds for any $u$, this shows $B = fI$ where $I$ is the identity map and $f$ is a smooth function. Since $B(du) = fdu$ is closed for any $u$, this shows $f$ must be locally constant.

With the help of Runge Approximation Theorem, one can show that the same result holds if the isomorphism maps all divergence free exact 1-forms to closed 1-forms.
\begin{lem}\label{lem: constant scaling}
    Let $M$ be a smooth connected manifold of dimension $n \geq 3$, $B$ is a fiber-wise isomorphism on $T^*M$ such that $B(T^*_xM) = T^*_xM$ for all $x \in M$. Suppose for every smooth $\psi$ satisfying
    \begin{equation}
        \Delta_g \psi = 0
    \end{equation}
    the image $B (d\psi)$ is a closed 1-form, then $B = CI$ for some constant $C \neq 0$ and $I$ is the identity map.
\end{lem}
\begin{proof}
    Fix some interior point $p$, let $x = (x^1, \dots, x^n)$ be the normal coordinate at $p$. That is, the metric is the Kronecker delta at $p$ and the Christoffel symbol vanishes at $p$. Let $H$ be any symmetric matrix such that $\tr(H) = 0$, note that at $p$ this is the same as $\tr_g(H) = 0$ in normal coordinates. by Proposition \ref{prop: runge}, there exists some solution $\psi$ such that $d\psi|_p = 0$ and $\text{Hess}_g(\psi)|_p = H$. Since we choose the normal coordinate, this simply means $\partial^2_{jk}\psi|_p = H_{jk}$. Use the assumption that $B(d\psi)$ is closed, at $p$ we have
    \begin{equation}
        0 = d(B(d\psi))|_p = d(B^j_k\partial_j\psi dx^k)|_p = B^j_kH_{jl}dx^l \wedge dx^k.
    \end{equation}
    Thus $BH$ is symmetric for any symmetric $H$ with zero trace. For all pairs of $i \neq j$, test $B$ with the following matrices:
    \begin{enumerate}
        \item $H_{ii} = 1$, $H_{jj} = -1$ and 0 everywhere else;
        \item $H_{ij} = H_{ji} = 1$ and 0 everywhere else.
    \end{enumerate}
    Together we obtain $B$ must be of the form $B = fI$ for some smooth function $f$. On the other hand, if we choose $\psi$ to be such that $\text{Hess}_g(\psi)|_p = 0$ and $d\psi|_p = dx^i$ for $i = 1, \dots, n$, then at $p$ we have
    \begin{equation}
        0 = d(f\partial_j\psi dx^j)|_p = \partial_lfdx^i \wedge dx^l.
    \end{equation}
    As a result, we obtain $f \equiv C$ for some constant $C \neq 0$.
\end{proof}


\subsection{Perturbation of the Magnetic Potential}\label{sec: proof of riemannian thm}

We first recall the boundary determination result from \cite{LU89}. The proposition has a specific construction formula, but we shall only state it as a deterministic result since it is all we need, for details we refer to the original paper.
\begin{prop}[{\cite[Proposition 1.3]{LU89}}]\label{prop: boundary determination}
    Suppose $M$ has dimension $n \geq 3$, consider any local coordinate $(x^1, \cdots, x^{n-1})$ on the boundary around some $p \in \partial M$. Then the Taylor series of $g$ at $p$ in boundary normal coordinate is determined by $\Lambda_g$.
\end{prop}

\begin{proof}[Proof of Theorem \ref{thm: riemannian perturb A}:]
    We first work with arbitrary metric $g$. Let $h$ be a smooth 1-form supported in the interior that will be determined later. Consider the perturbed Laplace equation
    \begin{equation}
        \Delta_{g, \varepsilon h}u_{\varepsilon} = 0,\quad u_{\varepsilon}|_{\partial M} = f.
    \end{equation}
    Then $u_{\varepsilon} = u+i\varepsilon v^h + O(\varepsilon^2)$, where $u$ and $v^h$ solves
    \begin{align}
        &\Delta_g u = 0, \quad u|_{\partial U} = f;\\
        &\Delta_g v = h(\nabla_g u) + \div_g(uh^{\sharp}), \quad v|_{\partial M} = 0,
    \end{align}
    where $\nabla_gu$ is the gradient of $u$ with respect to $g$. For notation simplicity, denote derivative of DN map as
    \begin{equation}
        N(f, h) = -i\partial_{\varepsilon}|_0 \Lambda_{g, \varepsilon h}(f) = \partial_{\nu} v|_{\partial M}.
    \end{equation}
    Note that we do not have the $\langle h, \nu\rangle$ term because $h$ supports in the interior. Denote $V_g$ the volume form on $M$ with respect to $g$, and $V_{g, \partial}$ the volume form on $\partial M$ with respect to $g|_{T\partial M}$. Using divergence theorem and the fact that $h$ supports in the interior, we have
    \begin{equation}
        \int_{\partial M} N(f, h) V_{g, \partial} = \int_{\partial M} \partial_{\nu}v V_{g, \partial} = \int_M \Delta_gv V_g = \int_M h(\nabla_gu)V_g.
    \end{equation}
    From Proposition \ref{prop: boundary determination}, the metric on the boundary is known, in particular this means the left hand side of the above equation is fully computable.

    We now choose $h$. For every fixed point $x_0$ in the interior let $x = (x^1, \cdots, x^n)$ be a local coordinate in the interior. Let
    \[
    \varphi^{\delta}(x) = \delta^{-n}\varphi\left(\frac{x-x_0}{\delta}\right)
    \]
    be a compactly supported approximate of identity at $x_0$ as $\delta \to 0$. Then for any smooth 1-form $\omega$, we have
    \begin{equation}
        \int_M \varphi^{\delta}\omega(\nabla_g u)|g|_x^{\frac{1}{2}} dx \xrightarrow{\delta \to 0} |g(x_0)|^{\frac{1}{2}}\omega(\nabla_g u)|_{x_0},
    \end{equation}
    where $|g|_x$ is the determinant of $g$ in coordinate $x$. This means for every initial data $f$ and every point $x_0$ in the interior, we can explicitly compute
    \begin{equation}
        \lim_{\delta \to 0}\int_{\partial M} N(f, \varphi_{x_0}^{\delta}\omega) V_{g,\partial} = |g(x_0)|^{\frac{1}{2}}\omega(\nabla_gu)|_{x_0}.
    \end{equation}
    Invariantly speaking, we can thus compute the coupled term $\nabla_gu V_g$. If $\Lambda_{g_1,A} = \Lambda_{g_2, A}$ for all small $A$, this means
    \begin{equation}\label{eq: same grad}
        \nabla_{g_1} u_1 V_{g_1} = \nabla_{g_2}u_2 V_{g_2}
    \end{equation}
    where $u_1$ and $u_2$ solves
    \begin{equation}
        \Delta_{g_j} u_j = 0, \quad u_j|_{\partial M} = f, \quad j = 1, 2.
    \end{equation}
    This gives a coordinate invariant relationship between $du_1$ and $du_2$, via
    \begin{equation}
        du_1 = \frac{|g_2|^{\frac{1}{2}}}{|g_1|^{\frac{1}{2}}}g_1g_2^{-1}du_2,
    \end{equation}
    where $g_2^{-1}du_2$ refers to the lift of $du_2$ to a vector field via $g_2$, and $g_1g_2^{-1}du_2$ is the lift of $g_2^{-1}du_2$ to a 1-form again via $g_1$. The equation is clearly coordinate independent, and in local coordinate we have
    \begin{equation}
        \partial_j u_1 = \frac{|g_2|^{\frac{1}{2}}}{|g_1|^{\frac{1}{2}}} (g_1)_{jl} (g_2)^{lk} \partial_k u_2 := B^k_j \partial_k u_2.
    \end{equation}
    The operator $B$ is thus a fiber-wise isomorphism on the cotangent bundle, mapping the differential of the solution of the Laplace equation to an exact one form.
    
    By Lemma \ref{lem: constant scaling}, $B = CI$ for some constant $C \neq 0$ and $I$ is the identity map.
    That is,
    \begin{equation}
        \frac{|g_2|^{\frac{1}{2}}}{|g_1|^{\frac{1}{2}}} g_1g_2^{-1} = CI.
    \end{equation}
    Since we assumed the dimension $n \geq 3$, take determinant both sides and simplify, we obtain
    \begin{equation}
        \frac{|g_2|^{\frac{1}{2}}}{|g_1|^{\frac{1}{2}}} = C^{\frac{n}{n-2}}.
    \end{equation}
    Hence
    \begin{equation}
        g_1g_2^{-1} = C^{-\frac{2}{n-2}}I.
    \end{equation}
    Use the fact that $g_1$ and $g_2$ agree on boundary tangential directions by Proposition \ref{prop: boundary determination}, $C = 1$. That is, $g_1 = g_2$.
\end{proof}


\section{Lorentzian Calder\'on problem}\label{sec: Lorentzian}

\subsection{Preliminaries}\label{sec: lorentzian preliminary}

We will be working on admissible Lorentzian manifolds, the definition is a slight modification of \cite[Definition 1.1]{HU19}.
\begin{defn}\label{def: admissible}
    Let $n \geq 3$. Let $(M,g)$ be a smooth connected $n$-dimensional Lorentzian manifold with non-empty boundary; thus, $g$ has signature $(-,+,\dots,+)$. We call $(M,g)$ \emph{admissible} if
    \begin{enumerate}
        \item there exists a smooth, proper, surjective function $\tau \colon M \to \mathbb{R}$, such that $d\tau$ is everywhere timelike;
        \item the boundary $\partial M$ is timelike, i.e.\ the induced metric $\bar{g} := g|_{T\partial M \times T\partial M}$ is Lorentzian;
        \item $\partial M$ is \emph{strictly null-convex}: if $\nu$ denotes the outward pointing unit normal vector field on $\partial M$, then
        \[
        II(V,V) = g(\nabla_V \nu, V) \geq 0 \tag{1.1}
        \]
        for all null vectors $V \in T_p \partial M$, where $II$ is the second fundamental form.
    \end{enumerate}
\end{defn}

We briefly introduce concepts in microlocal analysis, for details we refer to \cite{Hor71, DH72, Hor85III, Hor85IV, Dui96, Hin25}. Given a smooth submanifold $K$, we denote $N^*K$ its conormal bundle. A distribution $u$ is called a conormal distribution if locally it can be written as an oscillatory integral of the form
\[
u(x) = \int e^{i\phi(x, \theta)}a(x, \theta)d\theta
\]
where $\phi$ is the phase function locally parameterizing $N^*K$, and $a$ is a smooth symbol, see \cite{Hor71}. More generally, when $N^*K$ is replaced with some Lagrangian submanifold of $T^*M$, the distribution is called a Lagrangian distribution. An operator mapping functions on $X$ to functions on $Y$ is called a Fourier Integral Operator (FIO) if the corresponding Schwartz kernel is a Lagrangian submanifold of $T^*Y \times T^*X$ with respect to the twisted symplectic form. In this case the Lagrangian submanifold is also referred to as a canonical relation. When $X = Y$ and the Lagrangian submanifold is the diagonal of $T^*X \times T^*X$, we call the operator a pseudodifferential operator. The wavefront set of any Lagrangian distribution is a subset of the Lagrangian submanifold, see \cite{Hor85III, Hor85IV}.

For any operator $P$ acting on distributions on a Lorentzian or Riemannian manifold $(M, g)$, there exists a natural way to turn it into an operator acting on half-density valued distributions (for definition of half-density, see \cite{Hin25}). For any half-density valued distribution $\omega$, define
\[
\tilde{P} \omega = |g|^{\frac{1}{4}}P(|g|^{-\frac{1}{4}}\omega).
\]
Here $|g|^{\frac{1}{4}} = |g|^{\frac{1}{4}}_x|dx|^{\frac{1}{2}}$ where $|g|_x$ is the determinant of $g$ in local coordinate $x$, and $|dx|^{\frac{1}{2}}$ is the standard half-density in coordinate $x$. Note that this conjugation obviously depends on the metric, one can also conjugate $P$ with other half-densities to avoid this. The benefit of $\tilde{P}$ as an operator on half-density valued distribution is that if $P$ is an FIO, then the principal symbol of $\tilde{P}$, denoted by $\sigma(\tilde{P})$, can be invariantly defined modulo lower order terms as a half-density on the Lagrangian submanifold (for simplicity we ignore the Maslov bundle), see \cite{Hor71} for details.

When the Dirichlet data is a conormal distribution corresponding to some conormal bundle $N^*K$, then the solution is a Lagrangian distribution corresponding to the flowout Lagrangian of $N^*K$. The flowout Lagrangian refers to the union of the null-bicharacteristics originated from $N^*K \cap \mathcal{H}$ where $\mathcal{H}$ is the hyperbolic set (for definition of hyperbolic set and null-bicharacteristic, we refer to \cite[Section 24]{Hor85III}). We also refer to \cite{SY18} for an optics construction of the solution.

The DN map $\Lambda_{g,A,q}$ has been proved to contain both pseudodifferential part and Fourier Integral Operator part, see for example \cite{SY18} and \cite{YZ25}. The corresponding canonical relation for the FIO part is given by the powers of lens relation. Denote $\mathcal{H}$ the hyperbolic set and $\pi: T^*M|_{\partial M} \to T^*\partial M$. We can define the lens relation $L$ on a non-trapping admissible Lorentzian manifold as follows.
\begin{defn}{\label{def: lens relation}}
    Let $(M, g)$ be a non-trapping admissible Lorentzian manifold. For any $(x', \xi') \in \mathcal{H}$, there exists a unique inward pointing lightlike direction $\xi$ such that $\pi(x', \xi) = (x', \xi')$. Let $\Gamma$ be the unique future pointing null-bicharacteristic from $(x', \xi)$, and $(y', \eta)$ is the exit point of $\Gamma$. The lens relation is the map
    \[
    L: \mathcal{H} \to \mathcal{H}, \quad (x', \xi') \to \pi(y', \xi').
    \]
    The projection of $\Gamma$ on $M$ is a null-geodesic, and we will refer to this null-geodesic as the unique geodesic corresponding to $(y', \eta'; x', \xi')$. 
\end{defn}
For more details of lens relation and DN map on a Lorentzian manifold, see for example \cite{SY18} and \cite{YZ25}. Let $U, V \subset \partial M$ be disjoint, denote $\Lambda^{U,V}_{g,A,q}$ the DN map restricted to $V \times U$, that is the Dirichlet data $f$ in \eqref{eq: general wave equation} is supported in $U$ and
\[
\Lambda^{U,V}_{g,A,q}(f) = \partial_{\nu}u|_{V}.
\]
In \cite{YZ25}, the principal symbol of $\Lambda^{U,V}_{g,A,q}$ is computed.

\begin{prop}[{\cite[Proposition 3.5]{YZ25}}]\label{prop: principal symbol formula}
    View $\Lambda^{U_0,U_k}_{g,A,q}$ as an operator on half-density valued distributions via
    \[
    \Lambda^{U_0,U_k}_{g,A,q}(f|\bar{g}|^{1/4}) := (\Lambda^{U_0,U_k}_{g,A,q}f)|\bar{g}|^{1/4}.
    \]
    Suppose $L^k(x_0', \xi^{0\prime}) = (x_k', \xi^{k\prime})$ for some $k \geq 1$, $x_0' \in U_0$, $x_k' \in U_k$, and $L$ is the lens relation. Then microlocally in a conic neighborhood of $(x_k', \xi^{k\prime}; x_0', \xi^{0\prime}) \in T^*U_k \times T^*U_0$, the operator $\Lambda^{U_0, U_k}_{g,A,q}$ is an elliptic Fourier integral operator of order 1 whose associated canonical relation is given by
    \[
    \{(L^k(x', \xi'); x', \xi'): (x',\xi') \in T^*U_0 \cap \mathcal{H}\}.
    \]
    Its principal symbol is given by
    \[
    \sigma(\Lambda^{U_0,U_k}_{g,A,q})|_{(x_k',\xi^{k\prime}; x_0',\xi^{0\prime})} = \pm 2i(-1)^k\exp\left(i\int_{\gamma^b} A(\dot{\gamma}^b)\right)|\xi^k_n|^{1/2}|\xi^0_n|^{1/2}(L^{-k})^*,
    \]
    where
    \begin{itemize}
        \item $\gamma^b$ is the future pointing broken null-geodesic that is the projection of the future pointing broken bicharacteristic connecting
        \[
        (x_0', \xi^{0\prime}) \to (x_1', \xi^{1\prime}) \to \cdots \to (x_k', \xi^{k\prime})
        \]
        (note that $\int_{\gamma^b} A(\dot{\gamma}^b)$ is independent of parametrization of $\gamma^b$);
        \item the sign is $+$/$-$ if $(x_0', \xi^{0\prime})$ is a past/future pointing covector.
    \end{itemize}
\end{prop}
Finally, we record the well-known equation for computing principal symbols of wave equation which will be used in the proof of perturbation of the metric. If $\tilde{u}$ and $\tilde{v}$ are a half-density valued Lagrangian distribution, $P_g = |g|^{\frac{1}{4}}(\Box_g+q)|g|^{-\frac{1}{4}}$, and $P_g\tilde{u} = \tilde{v}$, then
\[
\mathcal{L}_{H_p}\sigma(\tilde{u}) = i\sigma(\tilde{v}),
\]
where $H_p$ is the Hamiltonian vector field with respect to the principal symbol of $P_g$ and $\mathcal{L}_{H_p}$ is the Lie derivative along the flowout Lagrangian. We refer to \cite{DH72} for more details.

\subsection{Perturbation of The Electromagnetic Potential}\label{sec: proof of theorem 1}

\begin{proof}[Proof of Theorem \ref{thm: lorentzian perturb A}]
    By \cite[Theorem 4.3]{SY18}, the lens relation $L$ can be recovered from the DN map. Let $L(x', \xi') = (y', \eta')$, then they are connected by a unique null-bicharacteristic, whose projection $\gamma$ is a null-geodesic. Note that perturbing $A$ does not change the trajectory of the null-geodesic, but only adds spin to it. In particular, by Proposition \ref{prop: principal symbol formula} this means for any $A \in \mathcal{B}_{\delta}(\mathcal{A})$,
    \[
    \left.\frac{\sigma(\Lambda_{g,A,q})}{\sigma(\Lambda_{g,\mathcal{A},q})}\right|_{(y', \eta';x',\xi')} = \exp\left(i\int_{\gamma}(A-\mathcal{A})(\dot{\gamma})\right).
    \]

    Suppose $\gamma$ does not pass through point $x$, then it is clear that for all $A$ such that $A - \mathcal{A}$ is supported in a sufficiently small neighborhood of $x$,
    \begin{equation}\label{eq: exp int equal 1}
        \left.\frac{\sigma(\Lambda_{g,A,q})}{\sigma(\Lambda_{g,\mathcal{A},q})}\right|_{(y', \eta';x',\xi')} = 1.
    \end{equation}
    On the other hand, suppose $x = \gamma(t_0)$ for some $t_0$, and let $U$ be some neighborhood of $x$, then $\gamma(t) \in U$ for $t \in [t_0 - \varepsilon, t_0 + \varepsilon]$ and some $\varepsilon > 0$. Consider some $\phi \in C^{\infty}_c(\mathbb{R})$ such that $\supp(\phi) \subset (-\varepsilon, \varepsilon)$ and $\int \phi = 1$. Let $A'$ be such that
    \[
    A'(\dot{\gamma}(t)) = \phi(t - t_0), \quad t \in (t_0 - \varepsilon, t_0 + \varepsilon).
    \]
    One can easily extend $A'$ to be a smooth 1-form supported in $U$, and divide by a sufficiently large number $C$ so that it is $\delta$ close to 0 in $C^k$ topology. In particular, this shows that $A = \mathcal{A} + A' \in \mathcal{B}_{\delta}(\mathcal{A}; U)$ satisfies
    \begin{equation*}
        \left.\frac{\sigma(\Lambda_{g,A,q})}{\sigma(\Lambda_{g,\mathcal{A},q})}\right|_{(y', \eta';x',\xi')} = \exp\left( \frac{i}{C}\int_{-\varepsilon}^{\varepsilon}\phi \right) \neq 1.
    \end{equation*}
    The two cases give the criterion of whether $\gamma$ passes through a point $x$, hence we proved \eqref{eq: determine gamma from A}.

    To show the set of such $A$ is generic in $\mathcal{B}_{\delta}(\mathcal{A}; U)$ when $x \in \gamma$ and $U$ is a neighborhood of $x$, we need to show it is open and dense. Suppose $A \in \mathcal{B}_{\delta}(\mathcal{A}; U)$ satisfies
    \begin{equation}\label{eq: exp int not 1}
        \left.\frac{\sigma(\Lambda_{g,A,q})}{\sigma(\Lambda_{g,\mathcal{A},q})}\right|_{(y', \eta';x',\xi')} \neq 1.
    \end{equation}
    then it is clear that it holds in a sufficiently small neighborhood of $A$ in $\tilde{B}_{\delta}(\mathcal{A}; U)$ with respect to $C^k$ topology. On the other hand, to show density, if $A$ satisfies \eqref{eq: exp int equal 1}, same argument shows that there exists some $B$ arbitrarily close to $A$ such that $\supp(B - A) \subset U$ and 
    \[
    \left.\frac{\sigma(\Lambda_{g,B,q})}{\sigma(\Lambda_{g,A,q})}\right|_{(y', \eta';x',\xi')} \neq 1.
    \]
    Hence one can find $B \in \mathcal{B}_{\delta}(\mathcal{A}; U)$ arbitrarily close to $A$ and \eqref{eq: exp int not 1} holds for $B$ and $\mathcal{A}$. 
    
    Finally, the conformal class can be constructed because every non-trapped null-geodesic can be constructed and the manifold is non-trapping.

\end{proof}

\begin{rem}\label{rem: ab effect}
    It is easy to see from the proof that when constructing the perturbation, we don't want $A - \mathcal{A}$ to be an exact 1-form supported in the interior, otherwise the integral along the geodesic would still give 0. Another way to see this is that if $A - \mathcal{A} = d\phi$ for some compactly supported $\phi$ in the interior, then the DN map is the same for $A$ and $\mathcal{A}$ by the gauge equivalence (see \cite[Lemma 2.2]{SY18}). A change by closed form, on the other hand may change the DN map, and one can use them to detect trajectory. This aligns with the Aharonov-Bohm effect in physics (see \cite{AB59, AB61}) , which roughly speaking states that the particle may experience a phase shift caused by the electromagnetic potential even when the induced electromagnetic field is 0. Mathematically, $A$ provides a connection on the principal $U(1)$-bundle, and $F = dA$ is the curvature of the connection. The Aharonov-Bohm effect is simply saying that the holonomy depends not just on the curvature, but also the topology of the path.
\end{rem}

\subsection{Perturbation of The Metric}\label{sec: proof of theorem 2}

\begin{proof}[Proof of Theorem \ref{thm: lorentzian perturb g}]
    Denote $g_{\epsilon} = g - \epsilon h$ for $h$ some compactly supported smooth symmetric 2-tensor that will be determined later. Let $u_{\epsilon}$ and $u$ be the solution to \eqref{eq: wave} with respect to $g_{\epsilon}$ and $g$, respectively. Then
    \[
    u_{\epsilon} = u + \epsilon v + O(\epsilon^2)
    \]
    with $v$ solves
    \begin{equation}\label{eq: equation for v}
        \begin{cases}
            (\Box_{g}+q_g)v = -|g|^{-\frac{1}{2}}\partial_j(|g|^{\frac{1}{2}}h_g^{jk}\partial_k u) + \frac{1}{2}\partial_j(\tr_g(h))g^{jk}\partial_ku - q^1u\\
            v|_{\partial M} = 0, \quad \partial_{\nu}v|_{\partial M} = \partial_{\epsilon}|_0\Lambda_{g_\epsilon}(f),
        \end{cases}
    \end{equation}
    where $q_{g_{\epsilon}} = q_g + \epsilon q^1 + O(\epsilon^2)$, $\tr_g(h) = \tr(g^{-1}h)$, and $h_g^{jk}$ is the $jk$-entry of the matrix $g^{-1}hg^{-1}$. Note that $q^1$ depends on $g$ and $h$, and since the potential depends locally on the metric, $\supp(q^1) \subset \supp(h)$. Let $P_g = |g|^{\frac{1}{4}}\Box_g|g|^{-\frac{1}{4}}$, $\tilde{u} = u|g|^{\frac{1}{4}}$ and $\tilde{v} = v|g|^{\frac{1}{4}}$, the equation for the half-densities is
    \begin{equation}\label{eq: equation for v tilde}
        (P_{g}+q_g)\tilde{v} = -h_g^{jk}\partial_k\partial_j u |g|^{\frac{1}{4}} + l.o.t.
    \end{equation}
    where $l.o.t.$ denotes the lower order terms, that is the terms that only contains $u$ or $\nabla u$. 
    
    Let $L(x', \xi') = (y', \eta')$ and $\gamma$ be the unique null-geodesic corresponds to $(y', \eta'; x', \xi')$. Let $K \subset \partial M$ be such that the conormal bundle $N^*K$ contains $(x', \xi')$. If $f = u|_{\partial M}$ is a conormal distribution with respect to $N^*K$, then $u$ is a Lagrangian distribution with canonical relation being the flowout Lagrangian from the hyperbolic part of $N^*K$ along the Hamiltonian vector fields (see for example \cite{HUZ22}, or the optics construction in \cite{SY18}). As a result, $v$ is a Lagrangian distribution on the same flowout Lagrangian (one can show this by a symbolic construction as in \cite{MU79}, see also \cite{YZ25}). In particular, the principal symbol satisfies the following transport equation
    \[
    \mathcal{L}_{H_p}\sigma(\tilde{v}) = ih_g^{jk}\xi_j\xi_k\sigma(\tilde{u})
    \]
    where recall that $H_p$ is the Hamiltonian vector field with respect to $g$ and $\mathcal{L}_{H_p}$ is the Lie derivative along the flowout Lagrangian. Let $w_1, \dots, w_n$ be linearly independent vector fields on the flowout Lagrangian obtained via pushforward along the Hamiltonian flow $\varphi_t$, that is
    \[
    w_j|_{\varphi_t(x, \xi)} = \varphi_{t*}(w_j|_{(x, \xi)}), \quad \mathcal{L}_{H_p} w_j = 0.
    \]
    Then we have
    \[
    H_p(\sigma(\tilde{v})(w_1, \cdots, w_n)) = (\mathcal{L}_{H_p}\sigma(\tilde{v}))(w_1, \cdots, w_n) = ih_g^{jk}\xi_j\xi_k\sigma(\tilde{u})(w_1, \cdots, w_n).
    \]
    Denote $\Gamma:[0, T] \to T^*M$ a null-bicharacteristic in the flowout Lagrangian, and let $\gamma$ be its projection, the principal symbol of $\tilde{v}$ satisfies
    \[
    \sigma(\tilde{v})|_{\Gamma(t)}(w_1, \cdots, w_n) = iC\int_{\gamma} (g^{-1}hg^{-1})(\dot{\gamma}^{\flat}, \dot{\gamma}^{\flat}) = iC\int_{\gamma}h(\dot{\gamma}, \dot{\gamma})
    \]
    where we used the fact that
    \[
    \sigma(\tilde{v})|_{\Gamma(0)}(w_1, \cdots, w_n) = 0, \quad \sigma(\tilde{u})|_{\Gamma(t)}(w_1, \cdots, w_n) = C \neq 0 \quad \forall t \in [0, T].
    \]
    Now suppose $x \in \gamma$, $U$ is some neighborhood of $x$ and $N^*K$ contains $(x', \xi')$. It is clear that a generic choice of smooth symmetric 2-tensor $h$ supported in $U$ would satisfy that
    \[
    \int_{\gamma}h(\dot{\gamma}, \dot{\gamma}) \neq 0.
    \]
    That is to say, for any Dirichlet data $f$ which is a conormal distribution singular at $(x', \xi')$, a generic choice of smooth symmetric 2-tensor $h$ supported in $U$ gives that $\partial_{\epsilon}|_0\Lambda_{g-\epsilon h}(f)$ is not smooth at $(y', \eta')$. 
    
    On the other hand, if $x \notin \gamma$, then for sufficiently small neighborhood $U$ of $x$ and sufficiently small $K$, the projection of the flowout Lagrangian from the hyperbolic part of $N^*K$ is disjoint from $U$. From \eqref{eq: equation for v} and the assumption that potential term depends locally on $g$,
    \[
    (\Box_g + q_g)v \in C^{\infty}(M^\circ), \quad v|_{\partial M} = 0.
    \]
    As a result, $\partial_{\epsilon}|_0\Lambda_{g-\epsilon h}(f)$ is smooth for any smooth symmetric 2-tensor $h$ supported in $U$.

    Finally, the conformal class can be constructed because every non-trapped null-geodesic can be constructed and the manifold is non-trapping.

\end{proof}

\begin{rem}
    An obvious choice of $h$ such that
    \[
    \int_{\gamma}h(\dot{\gamma}, \dot{\gamma}) = 0
    \]
    is when $h = fg$ where $f$ is any compactly supported function. This choice of $h$ corresponds to perturbation within the conformally equivalence class, which does not change the DN map when working with Yamabe equation \eqref{eq: Yamabe}. Note that it does not mean we can not use this $h$ for other types of potential terms, since for many of them conformal equivalence is not a gauge equivalence. This simply suggests that one should check lower order terms of \eqref{eq: equation for v} in this case. In general, a simple choice can be
    \[
    h = g\begin{pmatrix}
        \phi &&&\\
        &0&&\\
        &&\ddots&\\
        &&&0
    \end{pmatrix} = \begin{pmatrix}
        -\phi & 0\\
        0 & 0
    \end{pmatrix}
    \]
    for some generic compactly supported $\phi$, in any semi-geodesic normal coordinate (see for example \cite[Lemma 2.3]{SY18} or \cite{Pet69})
    \[
    g = -(dx^1)^2+g_{\alpha \beta}(x)dx^{\alpha} \otimes dx^{\beta}, \quad \alpha,\beta \geq 2.
    \]
\end{rem}

\section*{Declarations}
The author states that there is no conflict of interest, and there is no associated data used in the paper.

\bibliographystyle{abbrv}
\bibliography{2025Oct29}

@article {KLU18,
    AUTHOR = {Kurylev, Yaroslav and Lassas, Matti and Uhlmann, Gunther},
     TITLE = {Inverse problems for {L}orentzian manifolds and non-linear
              hyperbolic equations},
   JOURNAL = {Invent. Math.},
  FJOURNAL = {Inventiones Mathematicae},
    VOLUME = {212},
      YEAR = {2018},
    NUMBER = {3},
     PAGES = {781--857},
      ISSN = {0020-9910,1432-1297},
   MRCLASS = {35R30 (35L71 53C65 58J45)},
  MRNUMBER = {3802298},
MRREVIEWER = {Enno\ Pais},
       DOI = {10.1007/s00222-017-0780-y},
       URL = {https://doi.org/10.1007/s00222-017-0780-y},
}

@article {AFO22,
    AUTHOR = {Alexakis, Spyros and Feizmohammadi, Ali and Oksanen, Lauri},
     TITLE = {Lorentzian {C}alder\'on problem under curvature bounds},
   JOURNAL = {Invent. Math.},
  FJOURNAL = {Inventiones Mathematicae},
    VOLUME = {229},
      YEAR = {2022},
    NUMBER = {1},
     PAGES = {87--138},
      ISSN = {0020-9910,1432-1297},
   MRCLASS = {35L05 (53C50 58J32)},
  MRNUMBER = {4438353},
       DOI = {10.1007/s00222-022-01100-5},
       URL = {https://doi.org/10.1007/s00222-022-01100-5},
}

@article {SY18,
    AUTHOR = {Stefanov, Plamen and Yang, Yang},
     TITLE = {The inverse problem for the {D}irichlet-to-{N}eumann map on
              {L}orentzian manifolds},
   JOURNAL = {Anal. PDE},
  FJOURNAL = {Analysis \& PDE},
    VOLUME = {11},
      YEAR = {2018},
    NUMBER = {6},
     PAGES = {1381--1414},
      ISSN = {2157-5045,1948-206X},
   MRCLASS = {53B30 (35A27 35R01 35R30 58J45)},
  MRNUMBER = {3803714},
MRREVIEWER = {Enno\ Pais},
       DOI = {10.2140/apde.2018.11.1381},
       URL = {https://doi.org/10.2140/apde.2018.11.1381},
}

@article{HU19,
    author = {Hintz, Peter and Uhlmann, Gunther},
    title = {Reconstruction of Lorentzian Manifolds from Boundary Light Observation Sets},
    journal = {International Mathematics Research Notices},
    volume = {2019},
    number = {22},
    pages = {6949-6987},
    year = {2017},
    month = {02},
    issn = {1073-7928},
    doi = {10.1093/imrn/rnx320},
    url = {https://doi.org/10.1093/imrn/rnx320},
    eprint = {https://academic.oup.com/imrn/article-pdf/2019/22/6949/30945542/rnx320.pdf},
}

@article{HUZ22,
    author = {Hintz, Peter and Uhlmann, Gunther and Zhai, Jian},
    title = {An Inverse Boundary Value Problem for a Semilinear Wave Equation on Lorentzian Manifolds},
    journal = {International Mathematics Research Notices},
    volume = {2022},
    number = {17},
    pages = {13181-13211},
    year = {2021},
    month = {05},
    issn = {1073-7928},
    doi = {10.1093/imrn/rnab088},
    url = {https://doi.org/10.1093/imrn/rnab088},
    eprint = {https://academic.oup.com/imrn/article-pdf/2022/17/13181/45617509/rnab088.pdf},
}

@article {Hor71,
    AUTHOR = {H\"ormander, Lars},
     TITLE = {Fourier integral operators. {I}},
   JOURNAL = {Acta Math.},
  FJOURNAL = {Acta Mathematica},
    VOLUME = {127},
      YEAR = {1971},
    NUMBER = {1-2},
     PAGES = {79--183},
      ISSN = {0001-5962,1871-2509},
   MRCLASS = {58G15 (35S05 47G05)},
  MRNUMBER = {388463},
MRREVIEWER = {Yu.\ V.\ Egorov},
       DOI = {10.1007/BF02392052},
       URL = {https://doi.org/10.1007/BF02392052},
}

@article {DH72,
    AUTHOR = {Duistermaat, Johannes Jisse and H\"ormander, Lars},
     TITLE = {Fourier integral operators. {II}},
   JOURNAL = {Acta Math.},
  FJOURNAL = {Acta Mathematica},
    VOLUME = {128},
      YEAR = {1972},
    NUMBER = {3-4},
     PAGES = {183--269},
      ISSN = {0001-5962,1871-2509},
   MRCLASS = {58G15 (35S05 47G05)},
  MRNUMBER = {388464},
MRREVIEWER = {Yu.\ V.\ Egorov},
       DOI = {10.1007/BF02392165},
       URL = {https://doi.org/10.1007/BF02392165},
}

@book {Hor85III,
    AUTHOR = {H\"ormander, Lars},
     TITLE = {The analysis of linear partial differential operators. {III}},
    SERIES = {Grundlehren der mathematischen Wissenschaften [Fundamental
              Principles of Mathematical Sciences]},
    VOLUME = {274},
      NOTE = {Pseudodifferential operators},
 PUBLISHER = {Springer-Verlag, Berlin},
      YEAR = {1985},
     PAGES = {viii+525},
      ISBN = {3-540-13828-5},
   MRCLASS = {35-02 (35Sxx 47G05 58G15)},
  MRNUMBER = {781536},
MRREVIEWER = {Min\ You\ Qi},
}

@book {Hor85IV,
    AUTHOR = {H\"ormander, Lars},
     TITLE = {The analysis of linear partial differential operators. {IV}},
    SERIES = {Grundlehren der mathematischen Wissenschaften [Fundamental
              Principles of Mathematical Sciences]},
    VOLUME = {275},
      NOTE = {Fourier integral operators},
 PUBLISHER = {Springer-Verlag, Berlin},
      YEAR = {1985},
     PAGES = {vii+352},
      ISBN = {3-540-13829-3},
   MRCLASS = {35-02 (35Sxx 47G05 58G15)},
  MRNUMBER = {781537},
MRREVIEWER = {Min\ You\ Qi},
}

@book {Dui96,
    AUTHOR = {Duistermaat, Johannes Jisse},
     TITLE = {Fourier integral operators},
    SERIES = {Progress in Mathematics},
    VOLUME = {130},
 PUBLISHER = {Birkh\"auser Boston, Inc., Boston, MA},
      YEAR = {1996},
     PAGES = {x+142},
      ISBN = {0-8176-3821-0},
   MRCLASS = {58G15 (35S30 47G30 81Q20)},
  MRNUMBER = {1362544},
MRREVIEWER = {Alejandro\ Uribe},
}

@book {Pet69,
    AUTHOR = {Petrov, Aleksei Zinovyevich},
     TITLE = {Einstein spaces},
      NOTE = {Translated from the Russian by R. F. Kelleher},
 PUBLISHER = {Pergamon Press, Oxford-Edinburgh-New York},
      YEAR = {1969},
     PAGES = {xiii+411},
   MRCLASS = {53.78 (83.00)},
  MRNUMBER = {244912},
}

@article{AFO24,
  author    = {Spyros Alexakis and Ali Feizmohammadi and Lauri Oksanen},
  title     = {Lorentzian Calderón problem near the Minkowski geometry},
  journal   = {Journal of the European Mathematical Society},
  year      = {2024},
  note      = {Published online first},
}

@article{Las25,
author = {Lassas, Matti},
year = {2025},
title = {Introduction to inverse problems for non-linear partial differential equations},
doi = {10.48550/arXiv.2503.12448},
journal = {arxiv preprint}
}

@article {UZ21,
    AUTHOR = {Uhlmann, Gunther and Zhai, Jian},
     TITLE = {Inverse problems for nonlinear hyperbolic equations},
   JOURNAL = {Discrete Contin. Dyn. Syst.},
  FJOURNAL = {Discrete and Continuous Dynamical Systems. Series A},
    VOLUME = {41},
      YEAR = {2021},
    NUMBER = {1},
     PAGES = {455--469},
      ISSN = {1078-0947,1553-5231},
   MRCLASS = {35R30 (35A27 35L70)},
  MRNUMBER = {4182329},
       DOI = {10.3934/dcds.2020380},
       URL = {https://doi.org/10.3934/dcds.2020380},
}

@article {Esk07,
    AUTHOR = {Eskin, Gregory},
     TITLE = {Inverse hyperbolic problems with time-dependent coefficients},
   JOURNAL = {Comm. Partial Differential Equations},
  FJOURNAL = {Communications in Partial Differential Equations},
    VOLUME = {32},
      YEAR = {2007},
    NUMBER = {10-12},
     PAGES = {1737--1758},
      ISSN = {0360-5302,1532-4133},
   MRCLASS = {35R30 (35L20)},
  MRNUMBER = {2372486},
MRREVIEWER = {Sergey\ G.\ Pyatkov},
       DOI = {10.1080/03605300701382340},
       URL = {https://doi.org/10.1080/03605300701382340},
}

@article {Esk17,
    AUTHOR = {Eskin, Gregory},
     TITLE = {Inverse problems for general second order hyperbolic equations
              with time-dependent coefficients},
   JOURNAL = {Bull. Math. Sci.},
  FJOURNAL = {Bulletin of Mathematical Sciences},
    VOLUME = {7},
      YEAR = {2017},
    NUMBER = {2},
     PAGES = {247--307},
      ISSN = {1664-3607,1664-3615},
   MRCLASS = {35R30 (35L20)},
  MRNUMBER = {3671738},
MRREVIEWER = {Sergey\ G.\ Pyatkov},
       DOI = {10.1007/s13373-017-0100-2},
       URL = {https://doi.org/10.1007/s13373-017-0100-2},
}

@article {Bel87,
    AUTHOR = {Belishev, Michael I.},
     TITLE = {An approach to multidimensional inverse problems for the wave
              equation},
   JOURNAL = {Dokl. Akad. Nauk SSSR},
  FJOURNAL = {Doklady Akademii Nauk SSSR},
    VOLUME = {297},
      YEAR = {1987},
    NUMBER = {3},
     PAGES = {524--527},
      ISSN = {0002-3264},
   MRCLASS = {35R30 (35L05)},
  MRNUMBER = {924687},
MRREVIEWER = {V.\ M.\ Isakov},
}

@article {Bel07,
    AUTHOR = {Belishev, Michael I.},
     TITLE = {Recent progress in the boundary control method},
   JOURNAL = {Inverse Problems},
  FJOURNAL = {Inverse Problems. An International Journal on the Theory and
              Practice of Inverse Problems, Inverse Methods and Computerized
              Inversion of Data},
    VOLUME = {23},
      YEAR = {2007},
    NUMBER = {5},
     PAGES = {R1--R67},
      ISSN = {0266-5611,1361-6420},
   MRCLASS = {93-02 (00-02 35B37 47N20 49N45 58E25 74J25 78A70)},
  MRNUMBER = {2353313},
       DOI = {10.1088/0266-5611/23/5/R01},
       URL = {https://doi.org/10.1088/0266-5611/23/5/R01},
}

@article {Tat99,
    AUTHOR = {Tataru, Daniel},
     TITLE = {Unique continuation for operators with partially analytic
              coefficients},
   JOURNAL = {J. Math. Pures Appl. (9)},
  FJOURNAL = {Journal de Math\'ematiques Pures et Appliqu\'ees. Neuvi\`eme
              S\'erie},
    VOLUME = {78},
      YEAR = {1999},
    NUMBER = {5},
     PAGES = {505--521},
      ISSN = {0021-7824},
   MRCLASS = {35A05 (35B60)},
  MRNUMBER = {1697040},
MRREVIEWER = {Masafumi\ Yoshino},
       DOI = {10.1016/S0021-7824(99)00016-1},
       URL = {https://doi.org/10.1016/S0021-7824(99)00016-1},
}

@article {RZ98,
    AUTHOR = {Robbiano, Luc and Zuily, Claude},
     TITLE = {Uniqueness in the {C}auchy problem for operators with
              partially holomorphic coefficients},
   JOURNAL = {Invent. Math.},
  FJOURNAL = {Inventiones Mathematicae},
    VOLUME = {131},
      YEAR = {1998},
    NUMBER = {3},
     PAGES = {493--539},
      ISSN = {0020-9910,1432-1297},
   MRCLASS = {35A20 (35A10)},
  MRNUMBER = {1614547},
MRREVIEWER = {Masafumi\ Yoshino},
       DOI = {10.1007/s002220050212},
       URL = {https://doi.org/10.1007/s002220050212},
}

@article {BK92,
    AUTHOR = {Belishev, Michael I. and Kurylev, Yaroslav V.},
     TITLE = {To the reconstruction of a {R}iemannian manifold via its
              spectral data ({BC}-method)},
   JOURNAL = {Comm. Partial Differential Equations},
  FJOURNAL = {Communications in Partial Differential Equations},
    VOLUME = {17},
      YEAR = {1992},
    NUMBER = {5-6},
     PAGES = {767--804},
      ISSN = {0360-5302,1532-4133},
   MRCLASS = {58G25},
  MRNUMBER = {1177292},
MRREVIEWER = {Johan\ Tysk},
       DOI = {10.1080/03605309208820863},
       URL = {https://doi.org/10.1080/03605309208820863},
}

@article {Ste89,
    AUTHOR = {Stefanov, Plamen},
     TITLE = {Uniqueness of the multi-dimensional inverse scattering problem
              for time dependent potentials},
   JOURNAL = {Math. Z.},
  FJOURNAL = {Mathematische Zeitschrift},
    VOLUME = {201},
      YEAR = {1989},
    NUMBER = {4},
     PAGES = {541--559},
      ISSN = {0025-5874,1432-1823},
   MRCLASS = {35P25 (35L05 35R30 47F05)},
  MRNUMBER = {1004174},
MRREVIEWER = {Hideo\ Tamura},
       DOI = {10.1007/BF01215158},
       URL = {https://doi.org/10.1007/BF01215158},
}

@article {FIKO21,
    AUTHOR = {Feizmohammadi, Ali and Ilmavirta, Joonas and Kian, Yavar and
              Oksanen, Lauri},
     TITLE = {Recovery of time-dependent coefficients from boundary data for
              hyperbolic equations},
   JOURNAL = {J. Spectr. Theory},
  FJOURNAL = {Journal of Spectral Theory},
    VOLUME = {11},
      YEAR = {2021},
    NUMBER = {3},
     PAGES = {1107--1143},
      ISSN = {1664-039X,1664-0403},
   MRCLASS = {35R30},
  MRNUMBER = {4322032},
MRREVIEWER = {Yuchan\ Wang},
       DOI = {10.4171/jst/367},
       URL = {https://doi.org/10.4171/jst/367},
}

@article {Isa91,
    AUTHOR = {Isakov, Victor},
     TITLE = {An inverse hyperbolic problem with many boundary measurements},
   JOURNAL = {Comm. Partial Differential Equations},
  FJOURNAL = {Communications in Partial Differential Equations},
    VOLUME = {16},
      YEAR = {1991},
    NUMBER = {6-7},
     PAGES = {1183--1195},
      ISSN = {0360-5302,1532-4133},
   MRCLASS = {35R30 (35L20)},
  MRNUMBER = {1116858},
MRREVIEWER = {Florin\ Iacob},
       DOI = {10.1080/03605309108820794},
       URL = {https://doi.org/10.1080/03605309108820794},
}

@article {RS91,
    AUTHOR = {Ramm, A. G. and Sj\"ostrand, J.},
     TITLE = {An inverse problem of the wave equation},
   JOURNAL = {Math. Z.},
  FJOURNAL = {Mathematische Zeitschrift},
    VOLUME = {206},
      YEAR = {1991},
    NUMBER = {1},
     PAGES = {119--130},
      ISSN = {0025-5874,1432-1823},
   MRCLASS = {35R30 (35L05)},
  MRNUMBER = {1086818},
MRREVIEWER = {Plamen\ D.\ Stefanov},
       DOI = {10.1007/BF02571330},
       URL = {https://doi.org/10.1007/BF02571330},
}

@article {Yi24,
    AUTHOR = {Yi, Yuchao},
     TITLE = {Bounded time inverse scattering for semilinear {D}irac
              equation},
   JOURNAL = {Inverse Problems},
  FJOURNAL = {Inverse Problems. An International Journal on the Theory and
              Practice of Inverse Problems, Inverse Methods and Computerized
              Inversion of Data},
    VOLUME = {41},
      YEAR = {2025},
    NUMBER = {6},
     PAGES = {Paper No. 065024},
      ISSN = {0266-5611,1361-6420},
   MRCLASS = {99-06},
  MRNUMBER = {4922376},
       DOI = {10.1088/1361-6420/ade471},
       URL = {https://doi.org/10.1088/1361-6420/ade471},
}

@article{YZ25,
  title        = {The Dirichlet-to-Neumann map for Lorentzian Calder\'{o}n problems with data on disjoint sets},
  author       = {Yi, Yuchao and Zhang, Yang},
  year         = {2025},
  eprint       = {2505.13676},
  archivePrefix= {arXiv},
  primaryClass = {math.AP},
  url          = {https://arxiv.org/abs/2505.13676},
    journal = {arxiv preprint}
}

@article {UZ23,
    AUTHOR = {Uhlmann, Gunther and Zhang, Yang},
     TITLE = {An inverse boundary value problem arising in nonlinear
              acoustics},
   JOURNAL = {SIAM J. Math. Anal.},
  FJOURNAL = {SIAM Journal on Mathematical Analysis},
    VOLUME = {55},
      YEAR = {2023},
    NUMBER = {2},
     PAGES = {1364--1404},
      ISSN = {0036-1410,1095-7154},
   MRCLASS = {35R30 (35Q35)},
  MRNUMBER = {4580338},
       DOI = {10.1137/22M1492490},
       URL = {https://doi.org/10.1137/22M1492490},
}

@article {CLOP22,
    AUTHOR = {Chen, Xi and Lassas, Matti and Oksanen, Lauri and Paternain,
              Gabriel},
     TITLE = {Detection of {H}ermitian connections in wave equations with
              cubic non-linearity},
   JOURNAL = {J. Eur. Math. Soc. (JEMS)},
  FJOURNAL = {Journal of the European Mathematical Society (JEMS)},
    VOLUME = {24},
      YEAR = {2022},
    NUMBER = {7},
     PAGES = {2191--2232},
      ISSN = {1435-9855,1435-9863},
   MRCLASS = {35R30 (58J45)},
  MRNUMBER = {4413765},
       DOI = {10.4171/jems/1136},
       URL = {https://doi.org/10.4171/jems/1136},
}

@article {CLOP21,
    AUTHOR = {Chen, Xi and Lassas, Matti and Oksanen, Lauri and Paternain,
              Gabriel P.},
     TITLE = {Inverse problem for the {Y}ang-{M}ills equations},
   JOURNAL = {Comm. Math. Phys.},
  FJOURNAL = {Communications in Mathematical Physics},
    VOLUME = {384},
      YEAR = {2021},
    NUMBER = {2},
     PAGES = {1187--1225},
      ISSN = {0010-3616,1432-0916},
   MRCLASS = {53C07 (17B80 22E30)},
  MRNUMBER = {4259385},
MRREVIEWER = {Rachel\ Lash\ Maitra},
       DOI = {10.1007/s00220-021-04006-0},
       URL = {https://doi.org/10.1007/s00220-021-04006-0},
}

@article {OSSU24,
    AUTHOR = {Oksanen, Lauri and Salo, Mikko and Stefanov, Plamen and
              Uhlmann, Gunther},
     TITLE = {Inverse problems for real principal type operators},
   JOURNAL = {Amer. J. Math.},
  FJOURNAL = {American Journal of Mathematics},
    VOLUME = {146},
      YEAR = {2024},
    NUMBER = {1},
     PAGES = {161--240},
      ISSN = {0002-9327,1080-6377},
   MRCLASS = {58J32 (35L05 35R01 35R30 47G30 53C65)},
  MRNUMBER = {4691487},
MRREVIEWER = {Enno\ Pais},
       DOI = {10.1353/ajm.2024.a917541},
       URL = {https://doi.org/10.1353/ajm.2024.a917541},
}

@article {LUW18,
    AUTHOR = {Lassas, Matti and Uhlmann, Gunther and Wang, Yiran},
     TITLE = {Inverse problems for semilinear wave equations on {L}orentzian
              manifolds},
   JOURNAL = {Comm. Math. Phys.},
  FJOURNAL = {Communications in Mathematical Physics},
    VOLUME = {360},
      YEAR = {2018},
    NUMBER = {2},
     PAGES = {555--609},
      ISSN = {0010-3616,1432-0916},
   MRCLASS = {58J45 (35A27 35L71 35R30)},
  MRNUMBER = {3800791},
MRREVIEWER = {Jean-Marc\ Delort},
       DOI = {10.1007/s00220-018-3135-7},
       URL = {https://doi.org/10.1007/s00220-018-3135-7},
}

@article {SUW22,
    AUTHOR = {S\'a{} Barreto, Ant\^onio and Uhlmann, Gunther and Wang,
              Yiran},
     TITLE = {Inverse scattering for critical semilinear wave equations},
   JOURNAL = {Pure Appl. Anal.},
  FJOURNAL = {Pure and Applied Analysis},
    VOLUME = {4},
      YEAR = {2022},
    NUMBER = {2},
     PAGES = {191--223},
      ISSN = {2578-5885,2578-5893},
   MRCLASS = {35P25 (35R30 58J50)},
  MRNUMBER = {4496085},
       DOI = {10.2140/paa.2022.4.191},
       URL = {https://doi.org/10.2140/paa.2022.4.191},
}

@article {MU79,
    AUTHOR = {Melrose, Richard and Uhlmann, Gunther},
     TITLE = {Lagrangian intersection and the {C}auchy problem},
   JOURNAL = {Comm. Pure Appl. Math.},
  FJOURNAL = {Communications on Pure and Applied Mathematics},
    VOLUME = {32},
      YEAR = {1979},
    NUMBER = {4},
     PAGES = {483--519},
      ISSN = {0010-3640,1097-0312},
   MRCLASS = {58G17 (35S05 47G05)},
  MRNUMBER = {528633},
MRREVIEWER = {Bent\ E.\ Petersen},
       DOI = {10.1002/cpa.3160320403},
       URL = {https://doi.org/10.1002/cpa.3160320403},
}

@book{Hin25,
  author    = {Peter Hintz},
  title     = {An Introduction to Microlocal Analysis},
  year      = {2025},
  edition   = {1},
  publisher = {Springer Cham},
  series    = {Graduate Texts in Mathematics},
  isbn      = {978-3-031-90705-0},
  eisbn     = {978-3-031-90706-7},
  series_issn = {0072-5285},
  series_eissn = {2197-5612},
  pages     = {X, 285},
  url       = {https://doi.org/10.1007/978-3-031-90706-7}
}

@article {AB59,
    AUTHOR = {Aharonov, Y. and Bohm, D.},
     TITLE = {Significance of electromagnetic potentials in the quantum
              theory},
   JOURNAL = {Phys. Rev. (2)},
  FJOURNAL = {Physical Review. Series II},
    VOLUME = {115},
      YEAR = {1959},
     PAGES = {485--491},
      ISSN = {0031-899X,1536-6065},
   MRCLASS = {81.00},
  MRNUMBER = {110458},
MRREVIEWER = {C.\ A.\ Hurst},
}

@article {AB61,
    AUTHOR = {Aharonov, Y. and Bohm, D.},
     TITLE = {Further considerations on electromagnetic potentials in the
              quantum theory},
   JOURNAL = {Phys. Rev. (2)},
  FJOURNAL = {Physical Review. Series II},
    VOLUME = {123},
      YEAR = {1961},
     PAGES = {1511--1524},
      ISSN = {0031-899X,1536-6065},
   MRCLASS = {81.31},
  MRNUMBER = {136318},
MRREVIEWER = {C.\ A.\ Hurst},
}

@article {LU89,
    AUTHOR = {Lee, John M. and Uhlmann, Gunther},
     TITLE = {Determining anisotropic real-analytic conductivities by
              boundary measurements},
   JOURNAL = {Comm. Pure Appl. Math.},
  FJOURNAL = {Communications on Pure and Applied Mathematics},
    VOLUME = {42},
      YEAR = {1989},
    NUMBER = {8},
     PAGES = {1097--1112},
      ISSN = {0010-3640,1097-0312},
   MRCLASS = {35R30 (58G15 78A30)},
  MRNUMBER = {1029119},
MRREVIEWER = {Gottfried\ Anger},
       DOI = {10.1002/cpa.3160420804},
       URL = {https://doi.org/10.1002/cpa.3160420804},
}

@article {LLS20,
    AUTHOR = {Lassas, Matti and Liimatainen, Tony and Salo, Mikko},
     TITLE = {The {P}oisson embedding approach to the {C}alder\'on problem},
   JOURNAL = {Math. Ann.},
  FJOURNAL = {Mathematische Annalen},
    VOLUME = {377},
      YEAR = {2020},
    NUMBER = {1-2},
     PAGES = {19--67},
      ISSN = {0025-5831,1432-1807},
   MRCLASS = {58J32 (31C12 31C15)},
  MRNUMBER = {4099622},
MRREVIEWER = {Peter\ R.\ Popivanov},
       DOI = {10.1007/s00208-019-01818-3},
       URL = {https://doi.org/10.1007/s00208-019-01818-3},
}

@incollection {Cal80,
    AUTHOR = {Calder\'on, Alberto-P.},
     TITLE = {On an inverse boundary value problem},
 BOOKTITLE = {Seminar on {N}umerical {A}nalysis and its {A}pplications to
              {C}ontinuum {P}hysics ({R}io de {J}aneiro, 1980)},
     PAGES = {65--73},
 PUBLISHER = {Soc. Brasil. Mat., Rio de Janeiro},
      YEAR = {1980},
   MRCLASS = {35R30 (35K60)},
  MRNUMBER = {590275},
MRREVIEWER = {J.\ R.\ Cannon},
}

@article {LU01,
    AUTHOR = {Lassas, Matti and Uhlmann, Gunther},
     TITLE = {On determining a {R}iemannian manifold from the
              {D}irichlet-to-{N}eumann map},
   JOURNAL = {Ann. Sci. \'Ecole Norm. Sup. (4)},
  FJOURNAL = {Annales Scientifiques de l'\'Ecole Normale Sup\'erieure.
              Quatri\`eme S\'erie},
    VOLUME = {34},
      YEAR = {2001},
    NUMBER = {5},
     PAGES = {771--787},
      ISSN = {0012-9593},
   MRCLASS = {58J32},
  MRNUMBER = {1862026},
MRREVIEWER = {Yongzhi\ Xu},
       DOI = {10.1016/S0012-9593(01)01076-X},
       URL = {https://doi.org/10.1016/S0012-9593(01)01076-X},
}

@article {LTU03,
    AUTHOR = {Lassas, Matti and Taylor, Michael and Uhlmann, Gunther},
     TITLE = {The {D}irichlet-to-{N}eumann map for complete {R}iemannian
              manifolds with boundary},
   JOURNAL = {Comm. Anal. Geom.},
  FJOURNAL = {Communications in Analysis and Geometry},
    VOLUME = {11},
      YEAR = {2003},
    NUMBER = {2},
     PAGES = {207--221},
      ISSN = {1019-8385,1944-9992},
   MRCLASS = {58J32},
  MRNUMBER = {2014876},
MRREVIEWER = {Rodney\ Josu\'e\ Biezuner},
       DOI = {10.4310/CAG.2003.v11.n2.a2},
       URL = {https://doi.org/10.4310/CAG.2003.v11.n2.a2},
}

@article {DKSU09,
    AUTHOR = {Dos Santos Ferreira, David and Kenig, Carlos E. and Salo,
              Mikko and Uhlmann, Gunther},
     TITLE = {Limiting {C}arleman weights and anisotropic inverse problems},
   JOURNAL = {Invent. Math.},
  FJOURNAL = {Inventiones Mathematicae},
    VOLUME = {178},
      YEAR = {2009},
    NUMBER = {1},
     PAGES = {119--171},
      ISSN = {0020-9910,1432-1297},
   MRCLASS = {58J32 (35R30)},
  MRNUMBER = {2534094},
MRREVIEWER = {Sergey\ G.\ Pyatkov},
       DOI = {10.1007/s00222-009-0196-4},
       URL = {https://doi.org/10.1007/s00222-009-0196-4},
}

@incollection {Uhl14,
    AUTHOR = {Uhlmann, Gunther},
     TITLE = {30 years of {C}alder\'on's problem},
 BOOKTITLE = {S\'eminaire {L}aurent {S}chwartz---\'Equations aux
              d\'eriv\'ees partielles et applications. {A}nn\'ee 2012--2013},
    SERIES = {S\'emin. \'Equ. D\'eriv. Partielles},
     PAGES = {Exp. No. XIII, 25},
 PUBLISHER = {\'Ecole Polytech., Palaiseau},
      YEAR = {2014},
      ISBN = {978-2-7302-1626-5},
   MRCLASS = {35R30 (35-02 35J25)},
  MRNUMBER = {3381003},
}

@incollection {Sal13,
    AUTHOR = {Salo, Mikko},
     TITLE = {The {C}alder\'on problem on {R}iemannian manifolds},
 BOOKTITLE = {Inverse problems and applications: inside out. {II}},
    EDITOR = {Uhlmann, Gunther},
    SERIES = {Math. Sci. Res. Inst. Publ.},
    VOLUME = {60},
     PAGES = {167--247},
 PUBLISHER = {Cambridge Univ. Press, Cambridge},
      YEAR = {2013},
      ISBN = {978-1-107-03201-9},
   MRCLASS = {35R30 (35J25 35R01 58A10 58J99)},
  MRNUMBER = {3098658},
}

@article {DKLS16,
    AUTHOR = {Dos Santos Ferreira, David and Kurylev, Yaroslav and Lassas,
              Matti and Salo, Mikko},
     TITLE = {The {C}alder\'on problem in transversally anisotropic
              geometries},
   JOURNAL = {J. Eur. Math. Soc. (JEMS)},
  FJOURNAL = {Journal of the European Mathematical Society (JEMS)},
    VOLUME = {18},
      YEAR = {2016},
    NUMBER = {11},
     PAGES = {2579--2626},
      ISSN = {1435-9855,1435-9863},
   MRCLASS = {58J32 (35J25 35R01 35R30 78A48)},
  MRNUMBER = {3562352},
MRREVIEWER = {Jingzhi\ Tie},
       DOI = {10.4171/JEMS/649},
       URL = {https://doi.org/10.4171/JEMS/649},
}

@article {UZ22,
    AUTHOR = {Uhlmann, Gunther and Zhang, Yang},
     TITLE = {Inverse boundary value problems for wave equations with
              quadratic nonlinearities},
   JOURNAL = {J. Differential Equations},
  FJOURNAL = {Journal of Differential Equations},
    VOLUME = {309},
      YEAR = {2022},
     PAGES = {558--607},
      ISSN = {0022-0396,1090-2732},
   MRCLASS = {35R30},
  MRNUMBER = {4346704},
MRREVIEWER = {Zhiyuan\ Li},
       DOI = {10.1016/j.jde.2021.11.033},
       URL = {https://doi.org/10.1016/j.jde.2021.11.033},
}

@article {KV84,
    AUTHOR = {Kohn, Robert and Vogelius, Michael},
     TITLE = {Determining conductivity by boundary measurements},
   JOURNAL = {Comm. Pure Appl. Math.},
  FJOURNAL = {Communications on Pure and Applied Mathematics},
    VOLUME = {37},
      YEAR = {1984},
    NUMBER = {3},
     PAGES = {289--298},
      ISSN = {0010-3640,1097-0312},
   MRCLASS = {80A20 (35R30)},
  MRNUMBER = {739921},
MRREVIEWER = {J.\ R.\ Ockendon},
       DOI = {10.1002/cpa.3160370302},
       URL = {https://doi.org/10.1002/cpa.3160370302},
}

@article {KV85,
    AUTHOR = {Kohn, R. V. and Vogelius, M.},
     TITLE = {Determining conductivity by boundary measurements. {II}.
              {I}nterior results},
   JOURNAL = {Comm. Pure Appl. Math.},
  FJOURNAL = {Communications on Pure and Applied Mathematics},
    VOLUME = {38},
      YEAR = {1985},
    NUMBER = {5},
     PAGES = {643--667},
      ISSN = {0010-3640,1097-0312},
   MRCLASS = {35R30 (80A20)},
  MRNUMBER = {803253},
MRREVIEWER = {J.\ R.\ Ockendon},
       DOI = {10.1002/cpa.3160380513},
       URL = {https://doi.org/10.1002/cpa.3160380513},
}

@article {HT13,
    AUTHOR = {Haberman, Boaz and Tataru, Daniel},
     TITLE = {Uniqueness in {C}alder\'on's problem with {L}ipschitz
              conductivities},
   JOURNAL = {Duke Math. J.},
  FJOURNAL = {Duke Mathematical Journal},
    VOLUME = {162},
      YEAR = {2013},
    NUMBER = {3},
     PAGES = {496--516},
      ISSN = {0012-7094,1547-7398},
   MRCLASS = {35R30 (35J10 35J25 78A46)},
  MRNUMBER = {3024091},
MRREVIEWER = {Hideo\ Soga},
       DOI = {10.1215/00127094-2019591},
       URL = {https://doi.org/10.1215/00127094-2019591},
}

@article {AP06,
    AUTHOR = {Astala, Kari and P\"aiv\"arinta, Lassi},
     TITLE = {Calder\'on's inverse conductivity problem in the plane},
   JOURNAL = {Ann. of Math. (2)},
  FJOURNAL = {Annals of Mathematics. Second Series},
    VOLUME = {163},
      YEAR = {2006},
    NUMBER = {1},
     PAGES = {265--299},
      ISSN = {0003-486X,1939-8980},
   MRCLASS = {30C62 (35J25)},
  MRNUMBER = {2195135},
MRREVIEWER = {Daniel\ Faraco},
       DOI = {10.4007/annals.2006.163.265},
       URL = {https://doi.org/10.4007/annals.2006.163.265},
}

@article {SU87,
    AUTHOR = {Sylvester, John and Uhlmann, Gunther},
     TITLE = {A global uniqueness theorem for an inverse boundary value
              problem},
   JOURNAL = {Ann. of Math. (2)},
  FJOURNAL = {Annals of Mathematics. Second Series},
    VOLUME = {125},
      YEAR = {1987},
    NUMBER = {1},
     PAGES = {153--169},
      ISSN = {0003-486X,1939-8980},
   MRCLASS = {35R30 (86A20)},
  MRNUMBER = {873380},
MRREVIEWER = {P.\ Szeptycki},
       DOI = {10.2307/1971291},
       URL = {https://doi.org/10.2307/1971291},
}

\end{document}